\documentclass[12pt,oneside]{article}
\usepackage[T1]{fontenc}
\usepackage{a4}
\usepackage{mathrsfs}
\usepackage{amsmath}
\usepackage{amsfonts}
\usepackage{hyperref}
\usepackage{color}
\usepackage{amssymb}

\usepackage{graphicx}
\newtheorem{thm}{Theorem}
\newtheorem{lem}[thm]{Lemma}

\newtheorem{df}{Definition}
\newtheorem{obs}{Observation}

\newcommand{\bdf}{\begin{df} \begin{rm}}
\newcommand{\edf}{\end{rm} \end{df}}

\newenvironment{proof}{{\bf Proof.}}{\hspace*{\fill} \rule{2mm}{2mm} \par \hspace{0.1mm}}

\title{Bondage number of grid graphs}

\author{Magda Dettlaff$^{\dagger}$, Magdalena Lema\'{n}ska$^{\dagger}$ and
 Ismael G. Yero$^{\ddagger}$
\\
\\
$^{\dagger}${\small Department of Technical Physics and Applied Mathematics} \\ {\small Gda\'nsk University of Technology, ul. Narutowicza 11/12 80-233 Gda\'nsk, Poland }\\ {\small mdettlaff\@@mif.pg.gda.pl}, {\small magda\@@mifgate.mif.pg.gda.pl}
\\
$^{\ddagger}${\small Departamento de Matem\'aticas, Escuela Polit\'ecnica Superior de Algeciras}\\
{\small Universidad de C\'adiz,} {\small
Av. Ram\'on Puyol, s/n, 11202 Algeciras, Spain.} \\ {\small
ismael.gonzalez\@@uca.es}}

\date{}

\begin{document}

\maketitle

\begin{abstract} The bondage number $b(G)$ of a nonempty graph $G$ is the cardinality of a smallest set
of edges whose removal from $G$ results in a graph with domination number greater than the domination number of $G$. Here we study the bondage number of some grid-like graphs. In this sense, we obtain some bounds or exact values of the bondage number of some strong product and direct product of two paths.
\end{abstract}

{\it Keywords:} Domination; bondage number; strong product graphs; direct product graphs.

{\it AMS Subject Classification Numbers:}  05C12; 05C76.

{\section{Introduction}} Let $G=(V,E)$ be a connected undirected
graph with vertex set $V$ and edge set $E$. Given two vertices $u,v\in V$, the notation $u\sim v$ means that $u$ and $v$ are adjacent. The \emph{neighborhood}
of a vertex $v\in V$ in $G$ is the set $N_{G}(v)=\{u\in V\;:\;u\sim v\}$. For a set $X\subseteq V,$ \emph{the open
neighborhood} $N_{G}(X)$ is defined to be $\bigcup_{v\in
X}N_{G}(v)$ and \emph{the closed neighborhood}
$N_{G}[X]=N_{G}(X)\cup X.$

The \emph{degree} $d_{G}(v)$ of a vertex $v$ is the number of edges
incident to $v$, $d_{G}(v)=|N_{G}(v)|.$ The minimum and maximum
degrees among all vertices of $G$ are denoted by $\delta (G)$ and
$\Delta (G)$,  respectively. The {\it distance} $d_{G}(u,v)=d(u,v)$
between two vertices $u$ and $v$ in a connected graph $G$ is the
length of a shortest $(u-v)$ path in $G.$

A set $D\subseteq V$ is a {\it dominating set} of $G$ if
$N_{G}[D]=V$. The {\it domination number} of $G$, denoted
$\gamma(G)$, is the minimum cardinality of a dominating set in $G$. Any dominating set of cardinality $\gamma(G)$ is called a $\gamma$-set. For unexplained terms and symbols see \cite{ks}.

The \emph{bondage number} $b(G)$ of a nonempty graph $G$ with $E\not
= \emptyset$ is the minimum  cardinality among all sets of edges
$E'\subseteq E$ for which $\gamma (G-E')>\gamma (G)$. The domination
number of every spanning subgraph of a nonempty graph $G$ is at
least as great as $\gamma (G)$, hence the bondage number of a
nonempty graph is well defined. Bondage number was introduced by
Fink et al. \cite{fi} in $1990$. However, the early research on the
bondage number can be found in Bauer et al. \cite{bh}. In
\cite{bh,fi} was shown that every tree has bondage number equal to $1$ or
$2$. Hartnell and Rall \cite{hr1} proved that for the cartesian
product $G_n=K_n\square K_n$, $n>1$, we have
$b(G_n)=\frac{3}{2}\Delta$. Teschner \cite{te, te3, te2} also
studied the bondage number; for instance, in \cite{te3} he showed that $b(G)\leq
\frac{3}{2}\Delta (G)$ holds for any graph $G$ satisfying
$\gamma(G)\leq 3$. Moreover, the bondage number of planar graphs was
described in \cite{de, frv, ky}. Carlson and Develin \cite{de} showed
that the corona $G=H\circ K_1$ satisfies $b(G)=\delta (H)+1$. In
\cite{ksk} Kang et al. proved for discrete torus $C_n\square C_4$
that $b(C_n\square C_4)=4$ for any $n\geq 4$. Also, some relationships between the connectivity and the bondage number of graphs were studied in \cite{ls}. In \cite{mesh}, the exact values of bondage number of Cartesian product of two paths $P_n$ and $P_m$ have been determined for $m\leq 4$. For more results on bondage number of a graph we suggest the survey \cite{survey}.

The following two lemmas show general bounds for the bondage number of a~graph.

\begin{lem}{\em \cite{hr1}}
If $u$ and $v$ are a pair of adjacent vertices of a graph $G$, then
\[b(G)\leq d(u)+d(v)-1-|N(u)\cap N(v)|.\]
\label{lem2}
\end{lem}

\begin{lem}{\em (\cite{bh}, \cite{hr2})}
If $u$ and $v$ are two vertices of a graph $G$ such that $d(u,v)\leq 2$, then \[b(G)\leq d(u)+d(v)-1.\]
\label{lem1}
\end{lem}

\section{Bondage number of $P_n\boxtimes P_m$}

Let $G$ and $H$ be two graphs with the sets of vertices $V_1=\{v_1,v_2,\ldots ,v_n\}$ and $V_2=\{u_1,u_2,\ldots ,u_m\}$, respectively. The strong product of $G$ and $H$ is the graph $G\boxtimes H$ formed
by the vertices $V=\{(v_i,u_j)\;:\;1\le i\le n,\,1\le j\le m\}$ and
two vertices $(v_i,u_j)$ and $(v_k,u_l)$ are adjacent in $G\boxtimes
H$ if and only if ($v_i=v_k$ and $u_j\sim u_l$), ($v_i\sim v_k$ and
$u_j=u_l$) or ($v_i\sim v_k$ and $u_j\sim u_l$). In this section we will study the bondage number of the strong product of two paths $P_n$ and $P_m$ of order at least two. We begin by giving some observations and lemmas which will be useful into obtaining the bondage number of $P_n\boxtimes P_m$ for $n,m\ge 2$.

We will say that a graph $G$ without isolated vertices satisfies the
property $\mathcal{P}$ if it has a dominating set of minimum
cardinality $S=\{u_1,u_2,...,u_k\}$, $k=\gamma(G)$, such that
$N[u_i]\cap N[u_j]=\emptyset$ for every $i,j\in \{1,...,k\}$, $i\ne
j$. Now, let $\mathfrak{F}$ be the class of all graphs satisfying
property $\mathcal{P}$. Notice that for instance every path graph belongs to $\mathfrak{F}$.

\begin{obs}\label{obs-paths}
Let $\{v_1,v_2,...,v_n\}$ be the set of vertices of a path $P_n$ of
order $n$. Then \begin{itemize}
\item[\rm{(i)}] If $n=3t$, then there is only one dominating set $S$ of minimum cardinality
in $P_n$; it satisfies property $\mathcal{P}$ and it is
$S=\{v_2,v_5,...,v_{n-1}\}$.
\item[\rm{(ii)}] If $n=3t+1$, then there is only one dominating set $S$ of minimum cardinality
in $P_n$ satisfying property $\mathcal{P}$ and it is
$S=\{v_1,v_4,v_7,...,v_{n-3},v_{n}\}$.
\item[\rm{(iii)}] If $n=3t+2$, then there are only two dominating sets $S$ and $S'$ of minimum
cardinality in $P_n$ satisfying property $\mathcal{P}$ and they are
$S=\{v_2,v_5,...,v_{n-3},v_{n}\}$ and
$S'=\{v_1,v_4,v_7,...,v_{n-1}\}$.
\end{itemize}
\end{obs}

The following result from \cite{yero-dom} is useful
into studying the bondage number of $P_n\boxtimes P_m$.

\begin{lem}{\em \cite{yero-dom}}\label{dom-number-strong-2}
For any $n,m\ge 2$,
$$\gamma(P_{n}\boxtimes P_{m})=\gamma(P_n)\gamma(P_m)=\left\lceil\frac{n}{3}\right\rceil\left\lceil\frac{m}{3}\right\rceil.$$
\end{lem}

\begin{thm}\label{bond-strong-cotas}
For any $n,m\ge 2$,
$$1\le b(P_n\boxtimes P_m)\le 5.$$
\end{thm}

\begin{proof}
Since $n,m\ge 2$, we have that there are always two adjacent vertices $u,v$ in $P_n\boxtimes P_m$ such that $d(u)=3$, $d(v)\le 5$ and
$|N(u)\cap N(v)|=2$. So, the result follows by Lemma \ref{lem2}.
\end{proof}

Similarly to the case of Cartesian product, hereafter we will study the bondage number of $P_n\boxtimes P_m$ by making some cases.

\begin{thm}
If $(n=3t$ and $m=3r)$ or $(n=3t$ and $m=3r+2)$, then
$$b(P_{n}\boxtimes P_m)=1.$$
\end{thm}

\begin{proof}
Notice that if $n=3t$ and $m=3r$, then by Observation
\ref{obs-paths} (i) there exists only one dominating set of minimum
cardinality in $P_n$ and only one dominating set of minimum
cardinality in $P_m$ and they satisfy the property $\mathcal{P}$.
Thus, there exists only one dominating set $S$, of minimum
cardinality in $P_{n}\boxtimes P_m$; and it also satisfies the
property $\mathcal{P}$. So, every vertex outside of $S$ is dominated
by only one vertex from $S$. Therefore, by deleting any edge $e$ of
$P_{n}\boxtimes P_m$ between a~vertex of $S$ and other vertex
outside of $S$, we obtain that the domination number of
$P_{n}\boxtimes P_m-\{e\}$ is greater than the domination number of
$P_{n}\boxtimes P_m$.

On the other hand, let $V_1=\{u_1,u_2,...,u_n\}$ and
$V_2=\{v_1,v_2,...,v_m\}$ be the set of vertices of $P_n$ and $P_m$,
respectively. Since $n=3t$, by Observation \ref{obs-paths}
(i), we have that there is only one dominating set of minimum
cardinality in $P_n$ and it is $S_1=\{u_2,u_5,...,u_{n-1}\}$.
Moreover, since $m=3r+2$ we have that every
dominating set $S_2$ of minimum cardinality in $P_m$ satisfies either
\begin{itemize}
\item $v_1\in S_2$ and $v_2,v_3\notin S_2$,
\item or $v_2\in S_2$ and $v_1,v_3\notin S_2$.
\end{itemize}
So, every dominating set $S$ of minimum cardinality in $P_{n}\boxtimes
P_m$ contains either the vertex $(u_2,v_1)$ (in which case, $(u_2,v_2)$ is only dominated by $(u_2,v_1)$) or the vertex $(u_2,v_2)$ (in which case, $(u_2,v_1)$ is only dominated by $(u_2,v_2)$) and also $S$ does not contain the vertex $(u_2,v_3)$, neither any vertex of type $(u_1,v_j)$ or $(u_3,v_l)$, with $j,l\in \{1,...,m\}$.
Thus, if we delete the edge $e'=(u_2,v_1)(u_2,v_2)$ we obtain that
any dominating set of minimum cardinality in $P_{n}\boxtimes P_m$ is not a dominating set in $P_{n}\boxtimes P_m-\{e'\}$. Therefore,
$\gamma(P_{n}\boxtimes P_{m}-\{e'\})>\gamma(P_{n}\boxtimes P_{m})$.
\end{proof}

\begin{thm}
If $n=3t$ and $m=3r+1$, then $$b(P_{n}\boxtimes P_m)=2.$$
\end{thm}

\begin{proof}
Let $V_1=\{u_1,u_2,...,u_n\}$ and $V_2=\{v_1,v_2,...,v_m\}$ be the
set of vertices of $P_n$ and $P_m$, respectively. If $n=3t$, then by
Observation \ref{obs-paths} (i) we have that there is only one
dominating set $S_1$ of minimum cardinality in $P_n$, it satisfies
property $\mathcal{P}$ and it is $S_1=\{u_2,u_5,...,u_{n-1}\}$.
Also, every dominating set of minimum cardinality in $P_m$ contains either the
vertex $v_1$ or the vertex $v_2$.

Thus, in $P_n\boxtimes P_m$, we
have that for every dominating set $S$ of minimum cardinality it is satisfied either $(u_2,v_1)\in S$ or $(u_2,v_2)\in S$. Notice that no vertex of type $(u_1,v_j)$ or $(u_3,v_l)$ is contained in $S$, with $j,l\in \{1,...,m\}$.  Now, since the vertex $(u_1,v_1)$ is only dominated by the vertices $(u_2,v_1)$ or $(u_2,v_2)$ by deleting the edges
$(u_2,v_1)(u_1,v_1)$ and $(u_2,v_2)(u_1,v_1)$ we have that
$$\gamma(P_n\boxtimes
P_m-\{(u_2,v_1)(u_1,v_1),(u_2,v_2)(u_1,v_1)\})>\gamma(P_n\boxtimes
P_m).$$ Thus, $b(P_{n}\boxtimes P_m)\le 2$.

On the other hand, since $n=3t$ we have that every vertex belonging
to any dominating set $S$ of minimum cardinality in $P_{n}\boxtimes P_m$
has the form $(u_i,v_j)$ where $u_i\in S_1$ and $S_1$ is the only dominating set of minimum cardinality in $P_n$. Hence, $S$ is
formed by $t$ subsets $A_{l}$, $l\in \{2,5,...,n-4,n-1\}$, such that
$A_l$ is a dominating set of minimum cardinality in the suitable copy of $P_m$ in $P_n\boxtimes P_m$; and $A_l$ dominates all the vertices of $\{u_{l-1},u_l,u_{l+1}\}\times P_m$ in the graph $P_n\boxtimes P_m$. Notice that the vertices of $\{u_{l-1},u_l,u_{l+1}\}\times P_m$ are only dominated by such a set $A_l$ and also, every dominating set of minimum cardinality in $P_m$ dominates the vertices of $\{u_{l-1},u_l,u_{l+1}\}\times P_m$.

Since $m=3r+1$ we have that $\gamma(P_m)=\gamma(P_{m-1})+1$. So, if we delete any edge $e$ of $P_m$ and $B$ is a dominating set of minimum cardinality in $P_m$, then we can obtain another dominating set $B'$ of minimum cardinality in $P_m-\{e\}$ such that $|B'|=|B|$.

Now, let $(u_i,v_j)\in S$. Thus, $(u_i,v_j)\in
A_{l}$ for some $l\in \{2,5,...,n-4,n-1\}$, $A_l$ is a dominating set of minimum cardinality in the suitable copy of $P_m$ in $P_n\boxtimes P_m$; and $A_l$ dominates all the vertices of $\{u_{l-1},u_l,u_{l+1}\}\times P_m$ in the graph $P_n\boxtimes P_m$. So, if we delete any edge incident
to $(u_i,v_j)$, then there exists
another set $A'_{l}$ such that it is a dominating set of minimum
cardinality in $P_m$ and $|A_l|=|A'_l|$. As a consequence, $A'_l$ dominates all the vertices of $\{u_{l-1},u_l,u_{l+1}\}\times P_m$ and the set $S'=S-A_l+A'_{l}$ is also a dominating set of minimum cardinality in $P_{n}\boxtimes P_m$ with
$|S|=|S'|$. Therefore, $b(P_{n}\boxtimes P_m)\ge 2$ and the result
follows.
\end{proof}

The following simply observation will be useful into proving the next Theorem.

\begin{obs}
Let us denote by $\{u_1,u_2,\ldots ,u_{3t+1}\}$ and $\{v_1,v_2,\ldots ,v_{3r+2}\}$ the sets of vertices of the paths $P_n=P_{3t+1}$ and $P_m=P_{3r+2}$, respectively. For every vertex $u_i$ $(1\leq i\leq 3t+1)$ there is a $\gamma$-set $D_n$ in $P_n$ which
contains $u_i$ and for every vertex $v_j$ $(1\leq j \leq 3r+2)$, where $j\not
\equiv 0 \pmod 3$, there is a $\gamma$-set $D_m$ in $P_m$ such that $v_j\in D_m$. Moreover, one of each two consecutive vertices $v_i$, $v_{i+1}$, where $i\equiv 1\pmod 3$, belongs to $D_m$.
\label{obs1}
\end{obs}

\begin{thm}
If $n=3t+1$ and $m=3r+2$, then $$b(P_{n}\boxtimes
P_m)=3.$$
\label{p1}
\end{thm}

\begin{proof} Let $(u,v)$ be a vertex of degree three in $P_n\boxtimes P_m$ and let $e_1,\ e_2,\ e_3$ denote edges
incident with $(u,v)$. We remove edges $e_1,\ e_2,\ e_3$ from $P_n\boxtimes P_m$. Hence, every dominating
set of minimum cardinality in $P_n\boxtimes P_m-\{e_1,e_2,e_3\}$ contains the vertex $(u,v)$.
Thus,
\begin{align*}
\gamma (P_n\boxtimes P_m-\{e_1,e_2,e_3\})&\geq \gamma (P_n\boxtimes (P_m-\{v\}))+1\\&=\gamma (P_{3t+1}\boxtimes P_{3r+1})+1\\&=(t+1)(r+1)+1.
\end{align*}

On the other hand, by Lemma \ref{dom-number-strong-2} we have that $\gamma (P_n\boxtimes P_m)=(t+1)(r+1)$. So, we obtain that $b(P_n\boxtimes P_m)\leq 3$.

On the other side, we show that removing any two edges does not change the domination number.
Let us denote by $\{u_1,u_2,\ldots ,u_{3t+1}\}$ and $\{v_1,v_2,\ldots ,v_{3r+2}\}$ the sets of vertices of the paths $P_n=P_{3t+1}$ and $P_m=P_{3r+2}$, respectively and let $H_k=P_n\boxtimes \{v_k\}$, where $1\leq k\leq m$ ($H_k\approx P_n$). We denote
$C=\{(u_i,v_j),\ 1\leq i\leq 3t+1,\ 1\leq j \leq 3r+2,\ j\not\equiv 0(mod 3)\}$. From Observation \ref{obs1},
for every vertex in $C$ there is a $\gamma$-set in $P_n\boxtimes P_m$ containing this vertex.
Now, we remove two edges $e_1$ and $e_2$. Obviously it is enough to consider the cases that
$e_1=ab$ and $e_2=xy$ have at least one end-vertex in $C$  (without loss of generality, let $a\in C$ and $x\in C$).
 Let us denote by $D_m$, $D_n$ and $D$ $\gamma$-sets in
$P_m$, $P_n$ and $P_n\boxtimes P_m-\{e_1,e_2\}$, respectively. We use the notation $a=(u^a,v^a)$,
$b=(u^b,v^b)$, $x=(u^x,v^x)$, $y=(u^y,v^y)$. The set $D'_n=\{u_1,u_4,\ldots u_{n-3},u_{n}\}$
is one of $\gamma$-sets of $P_n$. Let us denote $v^a=v_k$ for $1\leq k\leq m$ and $v^x=v_l$ for $1\leq l\leq m$. Without loss of generality, we can assume that $d(v_1,v^a)\leq d(v_1,v^x)$ (it means $k\leq l$).
In the following cases we show that $\gamma (P_n\boxtimes P_m-\{e_1,e_2\})=|D|=\gamma (P_n\boxtimes P_m)$
which implies that $b(P_n\boxtimes P_m)\geq 3$.

\emph{Case 1.} If $b,y\in V-C$, then $m\geq 5$. We have the following subcases.

\emph{Subcase 1.1.} If $l\equiv 2\;(mod\;3)$, then we denote a common neighbor of $a$ and $b$ in $C$
by $c=(u^c,v^c) \in C$, where $v^c=v^a=v_k$. We can construct $D_m$ such that $v_k$ and $v_{l+2}$ belong to $D_m$. We choose $D_n$ satisfying $u^c\in D_n$. Thus $D=D_n\times D_m$.

\emph{Subcase 1.2.} If $l \equiv 1\;(mod\; 3)$, then we have the following subcases.

\emph{Subcase 1.2.1.} If $k\not =l$, then we denote a common neighbor of $a$ and $b$ in $C$
by $c=(u^c,v^c) \in C$, where $v^c=v^a=v_k$, and we denote by $z=(u^z,v^z)\in C$, where $v^z=v^x=v_l$, a common neighbor of
$x$ and $y$ in $C$. We can construct $D_m$ such that $v^c$ and $v^z$ belong to $D_m$. We choose $D_n^c$ and $D_n^z$ satisfying that $u^c\in D_n^c$ and $u^z\in D_n^z$. Thus, $D=(D'_n\times D_m)-(D'_n\times \{v^c\})\cup (D_n^c\times \{v^c\})-
(D'_n\times \{v^z\})\cup (D_n^z\times \{v^z\})$.

\emph{Subcase 1.2.2.} If $k=l$, then we choose $D_m$ such that $v_{k-2}\in D_m$. Hence $D=D'_n\times D_m$.

\emph{Case 2.} If $b\in C$ and $y\in V-C$, then we denote a common neighbor of $a$ and $b$ by $c=(u^c,v^c)\in C$ and by
$z=(u^z,v^z)\in C$ a common neighbor of $x$ and $y$, where $v^x=v^z=v_l$. So, we have the following subcases.

\emph{Subcase 2.1.} If $l\equiv 2\;(mod\; 3)$, then we construct
$D_m$ such that $v^c,v_{l+2}\in D_m$ and $D_n$ that $u^c\in D_n$ (by Observation \ref{obs1}). Finally,
$D=D_n\times D_m$.

\emph{Subcase 2.2.} If $l\equiv 1\;(mod\; 3)$, then we construct
$D_m$ such that $v^c,v^z\in D_m$. We choose $D_n^c$ and $D_n^z$ such that $u^c\in D_n^c$ and $u^z\in D_n^z$. Thus $D=(D'_n\times D_m)-(D'_n\times \{v^z\})\cup (D_n^z\times \{v^z\})-(D'_n\times \{v^c\})\cup (D_n^c\times \{v^c\})$.

\emph{Case 3.} If $b\in V-C$ and $y\in C$, then by symmetry it is similar to \emph{Case 2.}

\emph{Case 4.} If $b,y\in C$, then the vertex $v^a$ either lies on a path $P_m$ between $v_{3p}$ and $v_{3(p+1)}$
for some integer $p$, $1\leq p\leq r$ or $v^a\in \{v_1,v_2,v_{m-1},v_m\}$. In the first case we can choose $D_m$
such that $v_{3p-1}$, $v_k$ and $v_{3(p+1)+1}$ belong to $D_m$, otherwise $v_k, v_3\in D_m$ or $v_k, v_{m-3}\in D_m$. Let $k'$ be such that $k\not =k'$
and $3p<k'<3(p+1)$. So, we consider the next subcases.

\emph{Subcase 4.1.} If $v^av^x\not \in E(P_m)$ and $v^a\not =v^x$, then we denote by $c\in C$ a~common neighbor of $a$ and $b$, and also
$x$ and $y$ have common neighbor $z\in C$. Similarly like in \emph {Subcase 1.2.1} we construct $D$ containing $c$ and $z$.

\emph{Subcase 4.2.} If $v^a=v^x$ or $v^av^x\in E(P_m)$, then we consider the following cases.

\emph{Subcase 4.2.1.} If $e_1$ and $e_2$ are adjacent, then we denote by $w=(u^w,v^w)\in C$ a common neighbor $a$, $b$, $x$ and $y$  and we construct $D$ such that $w\in D$ in the following way: we choose $D_n^w$ and $D_m^w$ such that $u^w\in D_n^w$ and $v^w\in D_m^w$. Then we take $D=D_n^w\times D_m^w$.

\emph{Subcase 4.2.2.} If $e_1$ and $e_2$ are not adjacent, then let $A=\{u^a,u^b,u^x,u^y\}\subseteq V(P_n)$. We consider the following cases:

\emph{Subcase 4.2.2.1.} If $u^a=u^b=u_i$ and $u^x=u^y=u_j$, then we consider the following items,
\begin{itemize}
\item if $i=1$ and $j=2$ ($i=n-1$, $j=n$), then we choose $D_n$ such that $u_2,u_3\in D_n$ ($u_{n-2},u_{n-1}\in D_n$) and $D=D_n\times D_m$.
\item if $i=1$ and $j\geq 3$ ($j=n$, $i\leq n-2$), then we choose $D_n$ such that $u_2,u_{j-1}\in D_n$ ($u_{i+1},u_{n-1}\in D_n)$ and $D=D_n\times D_m$.
\item if $i>1$ and $j<n$, then there exists $D_n$ such that $|\{u_{i-1},u_{i+1}\}\cap D_n|=1$ and $|\{u_{j-1},u_{j+1}\}\cap D_n|=1$ (in particular $v_{i+1}\in D_n$ for $j=i+2$) and $D=D_n\times D_m$.
\end{itemize}

\emph{Subcase 4.2.2.2} If $u^a=u^b=u_i$ and $u^x\not =u^y$, then let us say $u^x=u_j$ and $u^y=u_{j+1}$. Now, if $a$, $b$, $x$ and $y$ have a common neighbor $w\in C$ we construct $D$ such that $w\in C$ similarly as in \emph{Subcase 4.2.1}. Else, we choose such a set $D_n$ containing a neighbor of $u^a$.
So, we construct $D'=D_n\times D_m$ and; if $x\in D'$ (or $y\in D'$), then we exchange it with $(u^x,v_{k'})$ (for $y$ with $(u^y,v_{k'})$). After these modifications we obtain $D$ from $D'$.

\emph{Subcase 4.2.2.3} If $u^a\not =u^b$ and $u^x=u^y$, then it is similar to \emph{Subcase 4.2.2.2}.

\emph{Subcase 4.2.2.4} If $u^a\not =u^b$ and $u^x\not =u^y$, then we consider three subcases:

\begin{itemize}
\item $|A|=2$ and $A=\{u_i,u_{i+1}\}$. If $u_1\in A$ ($u_n\in A$) we can construct $D_n$ such that $u_1,u_3\in D_n$ ($u_{n-2},u_n\in D_n$). Else, $u_{i-1},u_{i+2}\in D_n$. Thus, we take $D=D_n\times D_m$.

\item $|A|=3$ and $A=\{u_i,u_{i+1},u_{i+2}\}$. So, we choose $D_n$ such that $u_i,u_{i+3}\in D_n$ for $i< n-3$ and $D_n=D'_n$ for $i=n-3$. We construct $D'=D_n\times D_m$ and; if $x\in D'$, then we exchange it with $(u^x,v_{k'})$. We do the same for $a$, $b$ and $y$. After these modifications we obtain $D$ from $D'$.

\item If $|A|=4$, then we denote vertices $x$ and $y$ such that $u^x=u_j$ and $u^y=u_{j+1}$. Then we choose $D_n$ which contains $u^a$. We construct $D'=D_n\times D_m$ and; if $a\in D'$, then we exchange it with $(u^a,v_{k'})$. We do the same for $x$ and $y$. After these modifications we obtain $D$ from $D'$.
\end{itemize}
$\,$
\end{proof}

\begin{obs}\label{bond-number-dom-set}
Let $G$ be a graph. If there are $t$ disjoint dominating sets of
minimum cardinality in $G$, then $b(G)\ge
\left\lceil\frac{t}{2}\right\rceil$.
\end{obs}

\begin{thm}
If $n=3t+2$ and $m=3r+2$, then $b(P_{n}\boxtimes P_m)=2$.
\end{thm}

\begin{proof}
Since $n=3t+2$ and $m=3r+2$, by Observation \ref{obs-paths} (iii)
there are two disjoint dominating sets of minimum cardinality in
each path $P_n$ and $P_m$. Thus, there are four disjoint dominating
sets of minimum cardinality in $P_{n}\boxtimes P_m$. Hence, by
Observation \ref{bond-number-dom-set} we have that $b(P_{n}\boxtimes
P_m)\ge 2$.

On the other hand, since $n=3t+2$ and $m=3r+2$, by Lemma \ref{dom-number-strong-2} we have that $\gamma(P_{n}\boxtimes
P_m)=(t+1)(r+1)$. Hence, any dominating set $S$ of minimum
cardinality in $P_{n}\boxtimes P_m$ leads to a vertex partition
$\Pi=\{A_1,A_2,...,A_{(t+1)(r+1)}\}$ of the graph $P_{n}\boxtimes
P_m$ with $|A_i\cap S|=1$, for every $i\in \{1,...,(t+1)(r+1)\}$.
Moreover, there exist two vertices $u_i,u_{i+1}$ in $P_n$, two
vertices $v_j,v_{j+1}$ in $P_m$ (See Figure \ref{fig-1}) and a set
$A_l\in \Pi$ such that
$A_l=\{(u_i,v_j),(u_i,v_{j+1}),(u_{i+1},v_j),(u_{i+1},v_{j+1})\}$,
only one of the vertices of the set $A_l$ belongs to $S$ and such a
vertex also dominates the rest of vertices in $A_l$, which are not
dominated by any other vertex in $S$. Thus, by deleting the edges
$e=(u_i,v_j)(u_{i+1},v_{j+1})$ and $f=(u_{i+1},v_j)(u_{i},v_{j+1})$,
we have that the set $S$ is not a dominating set of $P_{n}\boxtimes
P_m-\{e,f\}$.

\begin{figure}[h]
\begin{center}
\includegraphics[width=0.7\textwidth]{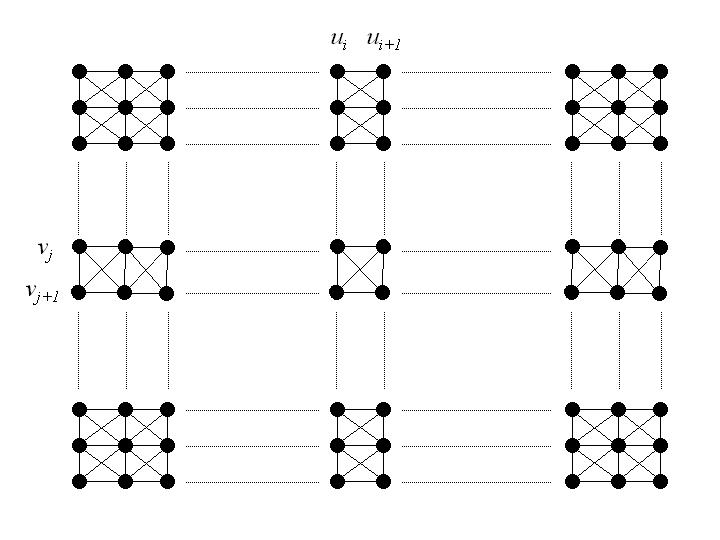}
\caption{The vertices $\{(u_i,v_j),
(u_i,v_{j+1}),(u_{i+1},v_j),(u_{i+1},v_{j+1})\}$.}\label{fig-1}
\end{center}
\end{figure}

Let us suppose there exists a set $S'$ with $|S'|=|S|$, such that
$S'$ is a dominating set in $P_{n}\boxtimes P_m-\{e,f\}$. Let
$\{x_1,x_2\}$ be the set of vertices of the path $P_2$ and let $H$
be the graph obtained from the graphs $P_n\boxtimes P_2$ and
$P_2\boxtimes P_m$, by identifying the vertices $(u_i,x_1)$,
$(u_i,x_2)$, $(u_{i+1},x_1)$ and $(u_{i+1},x_2)$ of $P_n\boxtimes
P_2$ with the vertices $(x_1,v_j)$, $(x_1,v_{j+1})$, $(x_2,v_j)$ and
$(x_2,v_{j+1})$ of $P_2\boxtimes P_m$, respectively (See Figure
\ref{fig-2}). Notice that $\gamma(H)=t+r+1$.

\begin{figure}[h]
\begin{center}
\includegraphics[width=0.7\textwidth]{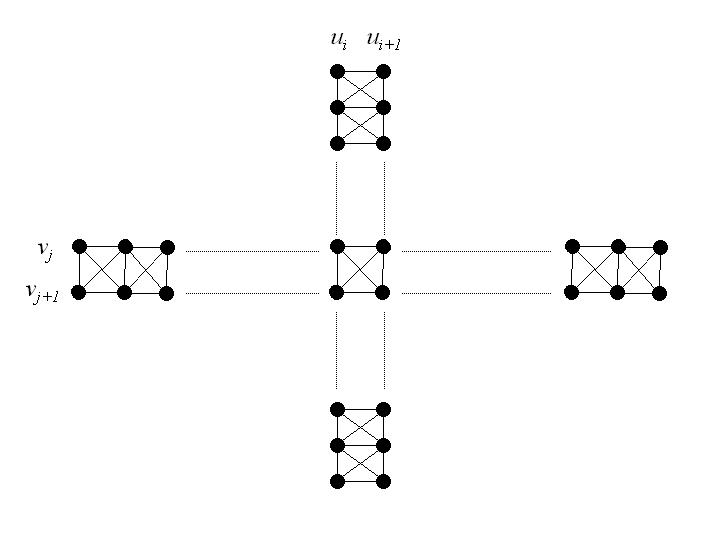}
\caption{The graph $H$.}\label{fig-2}
\end{center}
\end{figure}

Since $n=3t+2$ and $m=3r+2$, we have $$\gamma(P_{n}\boxtimes
P_m)=\gamma(P_{n-2}\boxtimes P_{m-2})+\gamma(H)=tr+t+r+1.$$ Hence, as $\gamma(H-\{e,f\})=t+r+2$ we
obtain that
\begin{align*}
\gamma(P_{n}\boxtimes P_m-\{e,f\})&= \gamma(P_{n-2}\boxtimes
P_{m-2})+\gamma(H-\{e,f\}) \\&=tr+t+r+2\\&>tr+t+r+1\\&=\gamma(P_{n}\boxtimes
P_m),
\end{align*}
which is a contradiction. Hence, there is no such a dominating set
$S'$ with $|S|=|S'|$ such that $S'$ dominates $P_{n}\boxtimes
P_m-\{e,f\}$. Therefore, the result follows.
\end{proof}

Finally, for the case $n=3t+1$ and $m=3r+1$, by Observation \ref{bond-number-dom-set} and Theorem \ref{bond-strong-cotas} we obtain the following bounds for the bondage number of  $P_n\boxtimes P_m$.

\begin{thm}
If $n=3t+1$ and $m=3r+1$, then $$2\le b(P_{n}\boxtimes
P_m)\le 5.$$
\end{thm}

Nevertheless we strongly think that in this case $b(P_n\boxtimes P_m)=5$.

\section{Bondage number of $P_n\times P_m$}

Let $G$ and $H$ be graphs with the sets of vertices $V_1=\{v_1,v_2,\ldots ,v_n\}$ and $V_2=\{u_1,u_2,\ldots ,u_m\}$, respectively. The direct product of $G$ and $H$ is the graph $G\times H$ formed by
the vertices $V=\{(v_i,u_j)\;:\;1\le i\le n,\,1\le j\le m\}$ and two
vertices $(v_i,u_j)$ and $(v_k,u_l)$ are adjacent in $G\times H$ if
and only if  $v_i\sim v_k$ and $u_j\sim u_l$. In this section we will study the bondage number of the direct product of two paths of order at least two.

Notice that any direct product of two paths contains at least two
vertices at distance two such that one of them has degree one and the other one has degree two. So,
Lemma \ref{lem1} leads to $b(P_n\times P_m)\le 2$.

\begin{thm}
For any paths $P_n$ and $P_m$,
\begin{itemize}
\item[{\rm (i)}] If $n\le 4$ or $m\le 4$, then $b(P_n\times P_m)=1$.
\item[{\rm (ii)}] If $n>4$ and $m> 4$, then $b(P_n\times P_m)\le 2$.
\end{itemize}
\end{thm}

\begin{proof}
(i) If $n\le 3$ or $m\le 3$, then there exist two vertices in $P_n\times
P_m$ at distance two such that they have degree equal to one. Thus,
by Lemma \ref{lem1} we obtain that $b(P_n\times P_m)=1$. If $n=m=4$, then $\gamma(P_4\times P_4)=4$ and
it is easy to verify that removing of any pendant edge leads to a graph $G'$ with $\gamma(G')=5$, what implies $b(P_4\times P_4)=1$.

(ii) On the contrary, if $n>4$ and $m> 4$, then there are two vertices in $P_n\times
P_m$ at distance two such that one of them has degree one and the other one has degree two. Thus, by Lemma \ref{lem1} we obtain that $b(P_n\times P_m)\le 2$.
\end{proof}

\begin{figure}[h]
\begin{center}
\includegraphics[width=0.4\textwidth]{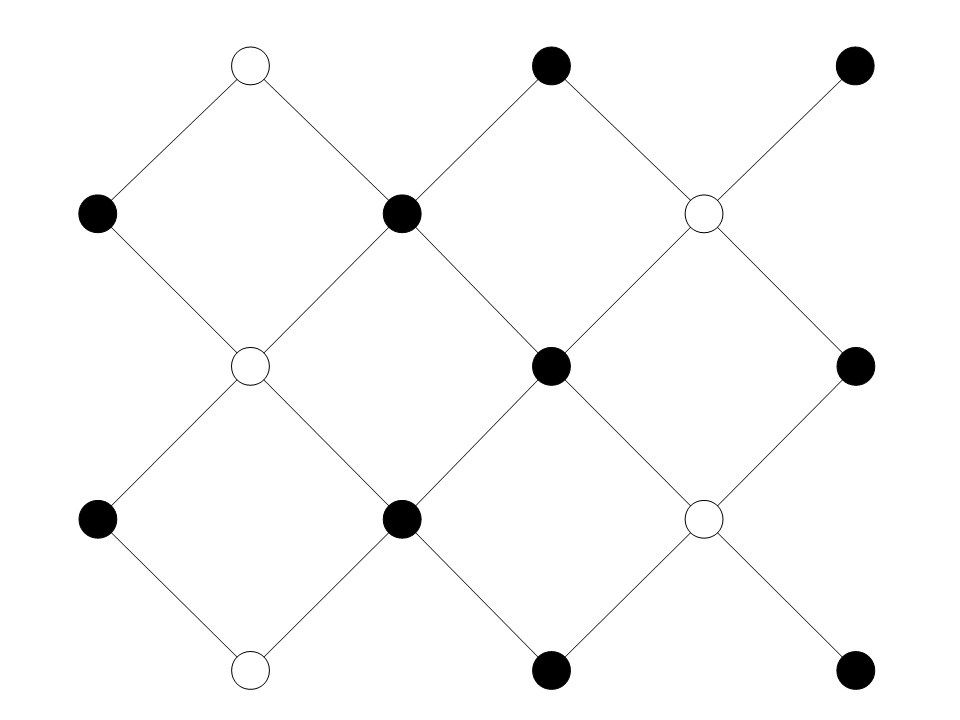}\hspace*{0.5cm}
\includegraphics[width=0.4\textwidth]{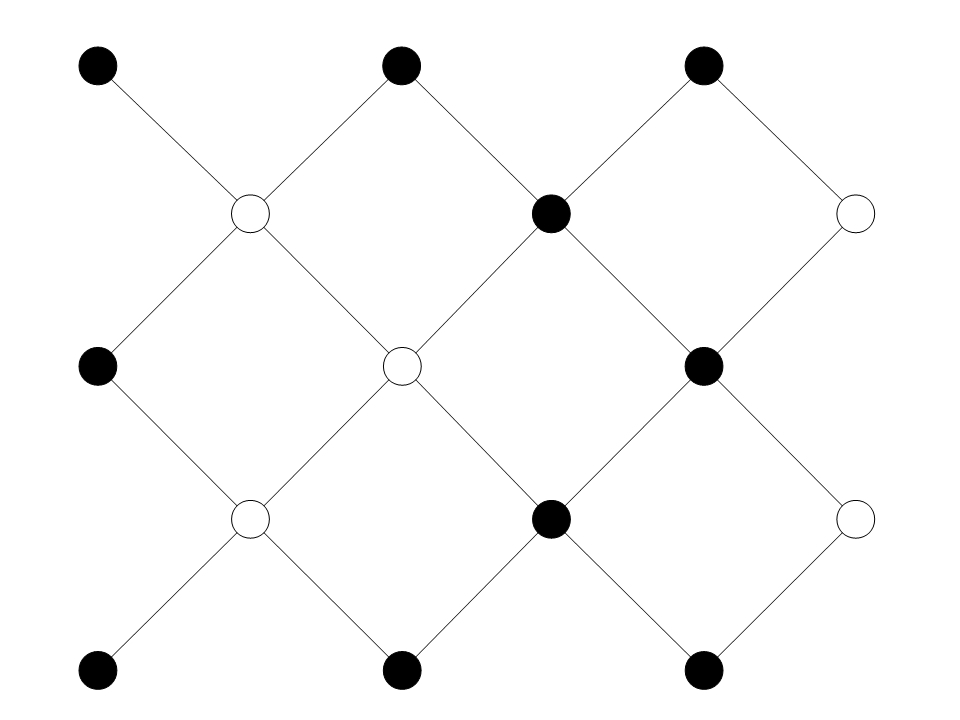}
\caption{The components $C_1$ and $C_2$ of $P_6\times P_5$.}\label{fig-3}
\end{center}
\end{figure}

Notice that there are values of $n,m\ge 4$ such that $b(P_n\times P_m)=2$. The graph $P_6\times P_5$ is an example, which has two isomorphic connected components $C_1$ and $C_2$ (See Figure \ref{fig-3}, where the vertices in white represents dominating sets of minimum cardinality in each component) having domination number equal to five. Thus, $\gamma(P_6\times P_5)=10$. Notice that by deleting any edge $e$ from $C_1$ or $C_2$ we can obtain a dominating set of cardinality five in $C_1-e$ or $C_2-e$. Therefore, we have that $b(P_6\times P_5)=2$.

\end{document}